\documentclass[12pt, reqno]{amsart}
\usepackage{amsmath, amsthm, amscd, amsfonts, amssymb, graphicx, color}
\usepackage[bookmarksnumbered, colorlinks, plainpages]{hyperref}
\usepackage[centertags]{amsmath}

\newcommand{\DD}{\hbox{D\kern-.73em\raise.25ex\hbox{-}\raise-.25ex\hbox{ }}}
\newcommand{\dD}{\hbox{d\kern-.33em\raise.75ex\hbox{-}\raise-.25ex\hbox{}}}
\newtheorem{theorem}{Theorem}[section]
\newtheorem{lemma}[theorem]{Lemma}

\newtheorem{corollary}[theorem]{Corollary}
\theoremstyle{definition}

\theoremstyle{remark}

\numberwithin{equation}{section}

\begin{document}
\setcounter{page}{1}

\title[On modulus of noncompact convexity ...]{On modulus of noncompact convexity for a strictly minimalizable measure of noncompactness}

\author[A. Reki\' c-Vukovi\' c, N. Oki\v ci\' c, I. Aran\DD elovi\' c]{Amra Reki\' c-Vukovi\' c$^1$ Nermin Oki\v ci\'c$^2$ $^{1}$
Ivan Aran\DD elovi\' c$^{2}$}

\address{$^{1}$ Department of Mathematics, University of Tuzla, Univerzitetska 4, Tuzla, Bosnia and Hercegovina.}
\email{\textcolor[rgb]{0.00,0.00,0.84}{amra.rekic@untz.ba}}\email{\textcolor[rgb]{0.00,0.00,0.84}{nermin.okicic@untz.ba}}

\address{$^{2}$ University of Belgrade - Faculty of Mechanical
Engineering, \\ Kraljice Marije 16, 11000 Belgrade, Serbia}
\email{\textcolor[rgb]{0.00,0.00,0.84}{iarandjelovic@mas.bg.ac.rs}}


\date{Received: xxxxxx; Revised: yyyyyy; Accepted: zzzzzz.}

\begin{abstract}
In this paper we consider modulus of noncompact convexity $\Delta_{X,\phi}$ associated with the strictly minimalizable
measure of noncompactness $\phi$. We also give some its properties and show its continuity on the interval
$[0,\phi(\overline{B}_X))$.
\end{abstract} \maketitle

\section{Introduction and preliminaries}

The theory of measures of noncompactness has many applications in Topology, Functional analysis and Operator theory.
There are many noneq\-ui\-valent definitions of this notion on metric and topological spaces
\cite{Akhmerov},\cite{Banas3}. First of them was introduced by Kuratowski in 1930. The most important examples of such
functions are: Kuratowski's measure ($\alpha$), Hausdorff's measure ($\chi$) and measure of Istratescu ($\beta$).

One of the tools that provides classification of Banach spaces considering their geometrical properties is the modulus
of convexity \cite{Clarkson}. Its natural generalization is the notion of noncompact convexity, introduced by K. Goebel
and T. Sekowski \cite{Goebel-Sekowski}. Their modulus was generated with Kuratowski's measure of noncompactness
($\alpha$). Modulus associated with  Hausdorff's measure of noncompactness ($\chi$), was introduced by Banas \cite
{Banas} and modulus associated with Istratescu's measure of noncompactness ($\beta$) by Dominguez-Benavides and Lopez
\cite{DL2}. In  \cite{Prus}, \cite{Ayerbe_knjiga} was presented an abstract unified to this notions, which consider
modulus of noncompact convexity $\Delta_{X,\phi}$ associated with arbitrary abstract measure of noncompactness $\phi$.

Banas \cite {Banas} proved that modulus $\Delta _{X,\chi }(\varepsilon)$ is subhomogeneous function in the case of
reflexive space $X$. Moreover, Prus \cite{Prus} gave the result connecting continuity of the modulus
$\Delta_{X,\phi}(\varepsilon)$ and uniform Opial condition which imply a normal structure of the space.

In this paper, we shall prove that modulus $\Delta _{X,\phi }(\varepsilon)$ is subhomogeneous and a continuous function
on the interval $[0,\phi (B_X))$, for an arbitrary strictly minimalizable measure of noncompactness $\phi$, where $X$
is Banach space having the Radon-Nikodym property.

\section{Preliminaries}

Let $X$ be Banach space, $B(x,r)$ denotes the open ball centered at $x$ with radius $r$, and $B_X$ and $S_X$ denote the
unit ball and sphere in the given space, respectively. If $A\subset X$ with $\overline{A}$ and $coA$ we denote closure
of the set $A$, that is convex hull of $A$.

Let $\mathfrak{B}$ be the collection of bounded subsets of the space $X$. Then function $\phi:\mathfrak{B}\rightarrow$
$[0,+\infty)$ with properties:

\begin{enumerate}

\item $\phi (B)=0\Leftrightarrow B$ is precompact set,

\item $\phi (B)=\phi (\overline{B})$, $\forall B\in$$\mathfrak{B}$,

\item $\phi (B_1\cup B_2)=\max \{ \phi (B_1),\phi (B_2)\}$, $\forall B_1,B_2\in$$\mathfrak{B}$,

\end{enumerate} is the measure of noncompactness defined on $X$.
For more about properties of the measure of noncompactness see in \cite{Akhmerov} and \cite{Ayerbe_knjiga}.

Let $\phi$ be a measure of noncompactness. Infinite set $A \in \mathfrak{B}$ is $\phi$-minimal if and only if $\phi(A)
= \phi(B)$ for any infinite set $B\subseteq A$.

We say that the measure of noncompactness $\phi$ is minimalizable if for every infinite, bounded set $A$ and for all
$\varepsilon>0$, there exists $\phi$-minimal set $B\subset A$, such that $\phi(B)\geq \phi(A)-\varepsilon$. Measure
$\phi$ is strictly minimalizable if for every infinite, bounded set $A$,
there exists $\phi$-minimal set $B\subset A$ such that $\phi(B)=\phi(A)$.\\
The modulus of noncompact convexity associated to arbitrary measure of noncompactness $\phi$ is the function
$\Delta_{X,\phi}:[0,\phi(B_X)]\rightarrow [0,1]$, defined with
$$\Delta_{X,\phi}(\varepsilon)=\inf \{ 1-d(0,A): A\subseteq \overline{B_X}, A=coA=\overline{A}, \phi(A)\geq \varepsilon\} \ .$$

Banas \cite{Banas} considered modulus $\Delta_{X,\phi}(\varepsilon)$ for $\phi=\chi $, where $\chi$ is Hausdorff
measure of noncompactness. For $\phi =\alpha$, $\alpha$ is Kuratowski measure of noncompactness,
$\Delta_{X,\alpha}(\varepsilon)$ presents Gobel-Sekowski modulus of noncompact convexity \cite{Goebel-Sekowski}, while
for $\phi=\beta$, where $\beta$ is a separation measure of noncompactness, $\Delta_{X,\beta}(\varepsilon)$ is
Dominguez Benavides -Lopez modulus of noncompact convexity \cite{DL2}.\\
The characteristic of noncompact convexity of $X$ associated to the measure of noncompactness $\phi$ is defined to be
$$\varepsilon_{\phi}(X)=\sup\{ \varepsilon\geq 0: \Delta_{X,\phi}(\varepsilon)=0\} \ .$$ Inequalities
$$\Delta_{X,\alpha}(\varepsilon)\leq \Delta_{X,\beta}(\varepsilon)\leq \Delta_{X,\chi}(\varepsilon) \ ,$$
hold for moduli $\Delta_{X,\phi}(\varepsilon)$ concerning $\phi = \alpha, \chi, \beta$, and consequently
$$\varepsilon _{\alpha} (X)\geq \varepsilon_{\beta} (X)\geq \varepsilon_{\chi}(X) \ .$$

Banach space $X$ has Radon-Nikodym property if and only if every nonempty, bounded set $A\subset X$ is dentable, that
is if and only if for all $\varepsilon >0$ there exists $x\in A$, such that $x\notin \overline{co}(A\backslash
\overline{B}(x,\varepsilon))$, \cite{Phelps}.

\section{Main results}

We started with our main result which is partial generalizations of earlier result obtained by Banas \cite{Banas}.

\begin{theorem}\label{subhomogenost}
Let $X$ be a Banach space with Radon-Nikodym property and $\phi$ strictly minimalizable measure of noncompactness. The
modulus $\Delta_{X,\phi}(\varepsilon)$ is subhomogeneous function, that is
\begin{equation}\label{th_subhomogenost_nejednakost}
\Delta_{X,\phi}(k\varepsilon)\leq k\Delta_{X,\phi}(\varepsilon),
\end{equation}
for all $k\in [0,1]$, $\varepsilon \in [0,\phi (\overline{B}_X)]$.
\end{theorem}
\begin{proof}
Let $\eta >0$ be arbitrary and $\varepsilon \in [0,\phi (\overline{B}_X)]$. From the definition of the modulus
$\Delta_{X,\phi }(\varepsilon)$ there exists convex, closed set $A\subseteq \overline{B}_X$, $\phi(A)\geq \varepsilon$
such that
\begin{equation}\label{th_subhomogenost_nejdn1}
1-d(0,A)< \Delta_{X,\phi}(\varepsilon)+\eta \ .
\end{equation}

The set $kA\subseteq \overline{B}_X$ is convex and closed for arbitrary $k\in [0,1]$ and $\phi (kA)=k\phi (A)\geq
k\varepsilon$. Since $\phi$ is strictly minimalizable measure of noncompactness there is infinite $\phi$-minimal set
$B\subset kA$ such that $$\phi (B)=\phi (kA)\geq k\varepsilon \ .$$ Let $n_0\in \mathbb{N}$ be such that $\displaystyle
\frac{k\varepsilon }{n_0}<\frac{diam B}{2}$. Since the set $B$ is bounded subset of Banach space $X$ which has
Radon-Nikodym property, we conclude that for $\displaystyle r=\frac{k\varepsilon}{n_0}$ there exists $x_0\in B$ such
that
$$x_0 \notin \overline{co}\left [B\backslash \overline{B}\left (x_0, r \right ) \right ] \ ,$$
that is
$$\overline{co}\left [B\backslash \overline{B}\left (x_0,r \right )\right ]\subset B\subset kA \ .$$
We consider the set $\displaystyle C=\overline{co}\left [B\backslash \overline{B}\left (x_0, r\right )\right ]$. $C$ is
closed and convex set and because of the properties of the strictly minimalizable measure of noncompactness we have
$$\phi (C)=\phi \left [B\backslash \overline{B}\left (x_0, r\right )\right ]=\phi (B)\geq k\varepsilon \ .$$
Moreover,
\begin{equation}\label{th_subhomogenost_nejdn2}
1-d(0,C)\leq 1-d(0,kA)=1-kd(0,A).
\end{equation}

We define the set $\displaystyle B^*=C+\frac{1-k}{\|x_0\|}x_0$. $B^*$ is a convex and closed set and arbitrary $z\in
B^*$ is of the form $\displaystyle z=y+\frac{1-k}{\|x_0\|}x_0$, where $y\in C\subset kA$ and $\|z\|< 1$. So,
$B^*\subset \overline{B}_X$. From the properties of the measure of noncompactness $\phi$ we have that
$$\phi (B^*)=\phi \left (C+\frac{1-k}{\|x\|}x\right )=\phi (C)\geq k\varepsilon \ .$$
Since $d(0,B^*)= d(0,C)+1-k$, than $1-d(0,B^*)=k-d(0,C)\leq k-kd(0,A)$ holds, so using inequality
(\ref{th_subhomogenost_nejdn1}) we have
$$1-d(0,B^*)<k(\Delta _{X,\phi }(\varepsilon)+\eta) \ .$$
If we take infimum by all sets $B$, such that $\phi (B)\geq k\varepsilon$, we have
$$\Delta _{X,\phi}(k\varepsilon)\leq k\Delta _{X,\phi }(\varepsilon)+ k\eta \ .$$
Since $\eta >0$ can be choosen arbitrarily small we obtain
$$\Delta _{X,\phi}(k\varepsilon)\leq k\Delta _{X,\phi }(\varepsilon) .$$
\end{proof}

As applications of Theorem \label{subhomogenost} we shall state the following corollaries.

\begin{corollary}
Let $\phi$ be a strictly minimalizable measure of noncompactness defined on space $X$ with Radon-Nikodym property. The
function $\Delta_{X,\phi}(\varepsilon)$ is strictly increasing on the interval
$[\varepsilon_{\phi}(X),\phi(\overline{B}_X)]$.
\end{corollary}
\begin{proof}
Let $\varepsilon_1$, $\varepsilon_2\in [\varepsilon_{\phi}(X),\phi(\overline{B}_X)]$ and $\varepsilon_1<\varepsilon_2$.
If we put $\displaystyle k=\frac{\varepsilon_1}{\varepsilon_2}<1$, then by the Theorem  \ref{subhomogenost} we obtain
$$\Delta_{X,\phi}(\varepsilon_1)=\Delta_{X,\phi}(k\varepsilon_2)\leq
k\Delta_{X,\phi}(\varepsilon_2)<\Delta_{X,\phi}(\varepsilon_2) \ .$$
\end{proof}

\begin{corollary}
Let $\phi$ be a strictly minimalizable measure of noncompactness defined on space $X$ with Radon-Nikodym property.
Inequality $$\Delta_{X,\phi}(\varepsilon)\leq \varepsilon \ $$ holds for all $\varepsilon \in
[0,\phi(\overline{B}_X)]$.
\end{corollary}
\begin{proof}
If $\varepsilon \in [0,1]$ and if $k$ and $\varepsilon$ change roles, and $\varepsilon $ take the value $\varepsilon
=1$, than by using Theorem \ref{subhomogenost}  we have $$\Delta_{X,\phi}(\varepsilon)\leq \varepsilon
\Delta_{X,\phi}(1)\leq \varepsilon \ .$$ If $\varepsilon \in (1,\phi(\overline{B}_X)]$, than the monotonicity of the
modulus $\Delta_{X,\phi}(\varepsilon)$ provides that
$$\Delta_{X,\phi}(\varepsilon)\leq \Delta_{X,\phi}(\phi(\overline{B}_X))=1<\varepsilon \ .$$
\end{proof}

\begin{corollary}
Let $\phi$ be a strictly minimalizable measure of noncompactness defined on the space $X$ with Radon-Nikodym property.
The function $\displaystyle f(\varepsilon)=\frac{\Delta_{X,\phi}(\varepsilon)}{\varepsilon}$ is nondecreasing on the
interval $[0,\phi(\overline{B}_X)]$ and for $\varepsilon_1+\varepsilon_2\leq \phi(\overline{B}_X)$ it holds
\begin{equation}\label{poslj3_nejednakost}
\Delta_{X,\phi}(\varepsilon_1+\varepsilon_2)\geq \Delta_{X,\phi}(\varepsilon_1)+\Delta_{X,\phi}(\varepsilon_2).
\end{equation}
\end{corollary}
\begin{proof}
Let $\varepsilon_1$, $\varepsilon_2\in [0,\phi (\overline{B}_X)]$ such that $\varepsilon_1\leq \varepsilon_2$. We put
$\displaystyle k=\frac{\varepsilon_1}{\varepsilon_2}$. Then
$$f(\varepsilon_1)=\frac{\Delta_{X,\phi}(\varepsilon_1)}{\varepsilon_1}=
\frac{\Delta_{X,\phi}(k\varepsilon_2)}{\varepsilon_1}\ .$$ If we use a property of subhomegeneity of the function
$\Delta_{X,\phi}(\varepsilon)$ we have $$f(\varepsilon_1)\leq
\frac{\Delta_{X,\phi}(\varepsilon_2)}{\varepsilon_2}=f(\varepsilon_2)\ ,$$ which proves that $f(\varepsilon)$ is a
nondecreasing function on the interval $[0,\phi(B_X)]$. Furthermore

 \begin{eqnarray*}
\Delta_{X,\phi}(\varepsilon_1)+\Delta_{X,\phi}(\varepsilon_2) & \leq & k\Delta_{X,\phi}(\varepsilon_2)+\Delta_{X,\phi}(\varepsilon_2) \\
 & = & \frac{\varepsilon _1+\varepsilon_2}{\varepsilon_2}\Delta_{X,\phi}(\varepsilon_2) \\
 & \leq & (\varepsilon _1+\varepsilon_2)\frac{\Delta_{X,\phi}(\varepsilon_1+\varepsilon _2)}{\varepsilon_1+\varepsilon_2}\\
 & = & \Delta_{X,\phi}(\varepsilon_1+\varepsilon _2).
\end{eqnarray*}
So, inequality (\ref{poslj3_nejednakost}) is proved.
\end{proof}

\begin{corollary}
Let $\phi$ be a strictly minimalizable measure of noncompactness defined on the space $X$ with Radon-Nikodym property.
\begin{equation}{\label{poslj4_nejednakost}}
\frac{\Delta_{X,\phi}(\varepsilon_2)-\Delta_{X,\phi}(\varepsilon_1)}{\varepsilon_2-\varepsilon_1}\geq
\frac{\Delta_{X,\phi}(\varepsilon_2)}{\varepsilon_2}.
\end{equation} holds for all $\varepsilon_1,\varepsilon_2\in (\varepsilon_1(X),\phi(\overline{B}_X)]$, $\varepsilon_1\leq \varepsilon_2$.
\end{corollary}
\begin{proof}
Let $\displaystyle k=\frac{\varepsilon_1}{\varepsilon_2}\leq 1$. From the Theorem \ref{subhomogenost} we have
\begin{eqnarray*}
 \Delta_{X,\phi}(\varepsilon_2)-\Delta_{X,\phi}(\varepsilon_1) & = &
 \Delta_{X,\phi}(\varepsilon_2)-\Delta_{X,\phi}(k\varepsilon_2) \\
  & \geq & \Delta_{X,\phi}(\varepsilon_2)-k\Delta_{X,\phi}(\varepsilon_2) \\
  & = & \frac{\varepsilon_2-\varepsilon_1}{\varepsilon_2}\Delta_{X,\phi}(\varepsilon_2).
 \end{eqnarray*}
Thus the inequality (\ref{poslj4_nejednakost}) is proved.
\end{proof}

\section{Continuity of the modulus of noncompact convexity}

In this section we shall prove continuity of the modulus of noncompact convexity associated with arbitrary strictly
minimalizable measure of noncompactness, defined on Banach space with Radon-Nikodym property. Now we need the following
Lemmas.

\begin{lemma}\label{th3.1_lema}
Let $\phi$ be an arbitrary measure of noncompactness, $B\subset \overline{B}_X$ $\phi$-minimal set and $k>0$ arbitrary.
Then the set $kB$ is $\phi$-minimal.
\end{lemma}
\begin{proof}
Let $B$ be $\phi$-minimal subset of closed unit ball and $D\subset kB$ arbitrary infinite set for $k>0$. If
$\displaystyle y\in \frac{1}{k}D$, then $\displaystyle y=\frac{1}{k}y'$ for some $y'\in D$. Since $y'\in kB$, then
$y'=kx$ for some $x\in B$. So $\displaystyle y=\frac{1}{k}kx=x\in B$, therefore $\displaystyle \frac{1}{k}D\subset B$.
Since $B$ is $\phi$-minimal set we have
$$\phi (B)=\phi \left (\frac{1}{k}D\right )=\frac{1}{k}\phi (D),$$ that is $\phi (D)=k\phi (B)=\phi (kB)$. Thus $kB$ is $\phi
$-minimal set.
\end{proof}

\begin{lemma} \label{th3.2_lema}
Let $X$ be a Banach space with Radon-Nikodym property and $\phi$ a strictly minimalizable measure of noncompactness.
The modulus of noncompact convexity $\Delta_{X,\phi}(\varepsilon)$ is left continuous function on the interval $[0,\phi
(\overline{B}_X))$.
\end{lemma}
\begin{proof}
Let $\varepsilon_0\in [0,\phi (\overline{B}_X))$ be arbitrary and let $\varepsilon <\varepsilon_0$. From the definition
of $\Delta_{X,\phi}(\varepsilon)$, for arbitrary $\eta >0$ there exists convex and closed set $A\subset
\overline{B}_X$, $\phi (A)\geq \varepsilon$ such that $$1-d(0,A)<\Delta_{X,\phi}(\varepsilon)+\eta \ .$$ Since $\phi$
is strictly minimalizable measure of noncompactness there exists infinite $\phi$-minimal set $B\subset A$ such that
$\phi (A)=\phi (B)\geq \varepsilon $. Moreover, inequality $$1-d(0,B)\leq 1-d(0,A)<\Delta_{X,\phi}(\varepsilon)+\eta \
$$ holds for the set $B$. Let $n_0\in \mathbb{N}$ be such that $\displaystyle \frac{\varepsilon}{n_0}<\frac{diam
B}{2}$. Since $B$ is a bounded subset of the Banach space $X$ with Radon-Nikodym property (\cite{Phelps}) we conclude
that for $\displaystyle r=\frac{\varepsilon }{n_0}$ there exists $x_0\in B$ such that $$\displaystyle x_0\notin
\overline{co}\left [B\backslash \overline{B}\left (x_0,r \right )\right ] \ .$$ This means that there is a convex and
closed set $\displaystyle C=\overline{co}\left [B\backslash \overline{B}\left (x_0,r \right )\right ] \subset A$, where
$$1-d(0,C)\leq 1-d(0,A)<\Delta_{X,\phi}(\varepsilon )+\eta \ .$$
Since $\phi (B)\geq \varepsilon $ we have that $\displaystyle B\backslash \overline{B}\left (x_0,r \right )$ is an
infinite set and using properties of the strictly minimalizable measure of noncompactness $\phi$ we obtain
$$\phi (C)=\phi \left [B\backslash \overline{B}\left (x_0,r \right )\right ]=\phi (B)\geq \varepsilon \ .$$
Let $\displaystyle k=1+\frac{1-d(0,C)}{2}$. We will consider set $\displaystyle A^*=kC \cap \overline{B}_X$. $A^*$ is a
closed and convex set such that $A^*\subseteq kC\subset kB$. From the Lemma \ref{th3.1_lema} we have $$\phi (A^*)=\phi
(kB)=k\phi (B)\geq k\varepsilon \ .$$ Moreover, inequality
$$1-d(0,A^*)\leq 1-d(0,kC)<1-d(0,C) \ ,$$
holds, i.e.
\begin{equation}{\label{th3.1_nejdn_1}}
1-d(0,A^*)< \Delta_{X,\phi}(\varepsilon )+\eta \ .
\end{equation}
Let $\displaystyle \delta =\varepsilon_0\left (1-\frac{1}{k}\right )$. For $\varepsilon \in
(\varepsilon_0-\delta,\varepsilon_0)$ we have
$$\phi (A^*)\geq k\varepsilon>k(\varepsilon_0-\delta)=k\frac{\varepsilon_0}{k}=\varepsilon _0 \ .$$
If we take infimum in (\ref{th3.1_nejdn_1}) by all the sets $A^*$ such that  $\phi (A^*)>\varepsilon _0$, we obtain
$$\inf \{ 1-d(0,A^*): A^*\subset \overline{B}_X, A^*=\overline{co}(A^*), \phi (A^*)>\varepsilon_0 \}\leq
\Delta_{X,\phi}(\varepsilon)+\eta \ ,$$ that is
$$\Delta_{X,\phi}(\varepsilon_0)\leq \Delta_{X,\phi}(\varepsilon)+\eta \ .$$ Since $\eta>0$ was arbitrarily small we conclude
$\displaystyle \lim_{\varepsilon \rightarrow
\varepsilon_{0_-}}\Delta_{X,\phi}(\varepsilon)=\Delta_{X,\phi}(\varepsilon_0) \ .$
\end{proof}

\begin{lemma} \label{th3.3_lema}
Let $X$ be a Banach space with Radon-Nikodym property and $\phi $ a strictly minimalizable measure of noncompactness.
The function $\Delta_{X,\phi}(\varepsilon)$ is right continuous on the interval $[0,\phi(\overline{B}_X))$.
\end{lemma}

\begin{proof}
Let $\eta >0$ and $\varepsilon _0\in [0,\phi (\overline{B}_X))$. From the definition of the modulus of noncompact
convexity $\Delta_{X,\phi}(\varepsilon)$, there exists convex and closed set $A\subset \overline{B}_X$, $\phi (A)\geq
\varepsilon _0$, such that $$1-d(0,A)<\Delta_{X,\phi}(\varepsilon _0)+\eta \ .$$ Since $\phi$ is strictly minimalizable
measure we conclude that there exists infinite, $\phi$-minimal set $B\subset A$ such that $\phi (A)=\phi (B)\geq
\varepsilon _0$ and
$$1-d(0,B)\leq 1-d(0,A)<\Delta_{X,\phi}(\varepsilon _0)+\eta \ .$$
The set $B$ is bounded subset of the Banach space that has Radon-Nikodym property, and because of that for
$\displaystyle r=\frac{\varepsilon_0}{n_0}$, where $n_0\in \mathbb{N}$ is such that $\displaystyle
\frac{\varepsilon_0}{n_0}<\frac{diam B}{2}$ we can find $x_0\in B$ such that $$\displaystyle x_0\notin
\overline{co}\left [B\backslash \overline{B}\left (x_0,r \right )\right] \ .$$ This implies that there is convex and
closed set $\displaystyle C=\overline{co}\left [B\backslash \overline{B}\left (x_0,r \right )\right ] \subset A$, where
$$1-d(0,C)\leq 1-d(0,A)<\Delta_{X,\phi}(\varepsilon_0 )+\eta \ .$$
Hence, by the properties of the strictly minimalizable measure of noncompactness $\phi$ we obtain
$$\phi (C)=\phi \left [B\backslash \overline{B}\left (x_0,r\right )\right ]=\phi (B)\geq \varepsilon_0 \ .$$
\begin{enumerate}
\item  If $\phi (C)>\varepsilon _0$, let $\delta =\phi (C)-\varepsilon_0$ and consider arbitrary $\varepsilon \in
(\varepsilon_0,\varepsilon_0+\delta)$. Then $\phi (C)>\varepsilon$, hence $$\inf \{ 1-d(0,C): C\subset \overline{B}_X,
C=\overline{co}C, \phi (C)>\varepsilon\}\leq \Delta_{X,\phi}(\varepsilon_0)+\eta \ .$$
    Moreover,
$$\Delta_{X,\phi}(\varepsilon)\leq \Delta_{X,\phi}(\varepsilon_0)+\eta \ .$$
This completes the proof of the Theorem.

\item Let $\phi (C)=\varepsilon_0$. Consider the set $\displaystyle B^*=kC \cap \overline{B}_X$ for $\displaystyle
k=1+\frac{1-d(0,C)}{2}$. $B^*$ is closed, convex set and it holds that $B^*\subseteq kC\subset kB$. Hence by the Lemma
\ref{th3.1_lema} we have
$$\phi (B^*)=\phi (kB)=k\phi (B)= k\varepsilon_0 \ .$$
Moreover, next inequality holds
$$1-d(0,B^*)\leq 1-d(0,kC)<1-d(0,C) \ .$$
Let $\delta =\varepsilon_0(1-k)$. Then for $\varepsilon \in (\varepsilon_0,\varepsilon _0+\delta)$ we have
$\phi(B^*)=k\varepsilon_0>\varepsilon,$ thus
$$\inf \{ 1-d(0,B^*): B^*\subset \overline{B}_X, B^*=\overline{co}B^*, \phi (B^*)>\varepsilon\}\leq \Delta_{X,\phi}(\varepsilon_0)+\eta \ ,$$
that is
$$\Delta_{X,\phi}(\varepsilon)\leq \Delta_{X,\phi}(\varepsilon_0)+\eta \ .$$
Therefore, we obtain $\displaystyle \lim_{\varepsilon \rightarrow
\varepsilon_{0_+}}\Delta_{X,\phi}(\varepsilon)=\Delta_{X,\phi}(\varepsilon_0)$. This completes the proof.
\end{enumerate}
\end{proof}

From Lemma \ref{th3.2_lema} and Lemma \ref{th3.3_lema} follows that:

\begin{theorem}
Let $X$ be a Banach space with Radon-Nikodym property and $\phi$ a strictly minimalizable measure of noncompactness.
The modulus $\Delta_{X,\phi}(\varepsilon)$ is a continuous function on the interval $[0,\phi (\overline{B}_X))$.
\end{theorem}

It is known that the Hausdorff measure of noncompactness $\chi$ is a strictly minimalizable measure of noncompactness
in the weakly compactly generated Banach spaces (\cite{Ayerbe_knjiga}, Theorem III. 2.7.). Since reflexive spaces are
weakly compactly generated and also have Radon-Nikodym property \cite{RN}, we conclude that the modulus of noncompact
convexity $\Delta_{X,\chi}$, associated to measure of noncompactness $\chi$, is a continuous function on the interval
$[0,\phi (\overline{B}_X))$ in an arbitrary weakly compactly generated space $X$.

{\bf Acknowledgement.} The third author was partially supported by Ministry of Education, Science and Technological
Development of Serbia, Grant No. 174002, Serbia.

\end{document}